\documentclass[11pt]{amsart}

\usepackage{amscd,amssymb,amsopn,amsmath,amsthm,mathrsfs,graphics,amsfonts,enumerate,verbatim,calc
}

\usepackage[all]{xy}

\usepackage{color}

\usepackage[OT2,OT1]{fontenc}
\newcommand\cyr{%
\renewcommand\rmdefault{wncyr}%
\renewcommand\sfdefault{wncyss}%
\renewcommand\encodingdefault{OT2}%
\normalfont
\selectfont}
\DeclareTextFontCommand{\textcyr}{\cyr}

\usepackage{amssymb,amsmath}

\DeclareFontFamily{OT1}{rsfs}{}
\DeclareFontShape{OT1}{rsfs}{n}{it}{<-> rsfs10}{}
\DeclareMathAlphabet{\mathscr}{OT1}{rsfs}{n}{it}

\numberwithin{equation}{section}
\newtheorem{theorem}{Theorem}[section]
\newtheorem{lemma}[theorem]{Lemma}
\newtheorem{proposition}[theorem]{Proposition}
\newtheorem{corollary}[theorem]{Corollary}

\newtheorem{question}{Question}

\theoremstyle{definition}
\newtheorem{definition}[theorem]{Definition}
\newtheorem{remark}[theorem]{Remark}
\theoremstyle{remark}

\newtheorem{example}[theorem]{Example}

\newtheorem{notation}[theorem]{Notation}
\newtheorem{acknowledgement}{Acknowledgement}

\newcommand{\depth}{\operatorname{depth}}

\newcommand{\fm}{\frak{m}}

\newcommand{\fq}{\frak{q}}

\begin{document}
\title[On the index of reducibility of parameter ideals]{On the index of reducibility of parameter ideals:\\ The stable and limit values}

\author[N.T. Cuong]{Nguyen Tu Cuong}
\address{Institute of Mathematics, Vietnam Academy of Science and Technology, 18 Hoang Quoc Viet Road, 10307
Hanoi, Vietnam}
\email{ntcuong@math.ac.vn}


\author[P.H. Quy]{Pham Hung Quy}

\address{Department of Mathematics, FPT University, Hanoi, Vietnam, and Thang Long Institute of Mathematics and Applied Sciences, Thang Long University, Hanoi, Vietnam}
\email{quyph@fe.edu.vn}

\thanks{
N.T. Cuong and P.H. Quy are partially supported by a fund of Vietnam National Foundation for Science
and Technology Development (NAFOSTED) under grant number 101.04-2017.10.}


\maketitle
\begin{center}
{\it On the occasion of Le Van Thiem's centenary}
\end{center}
\begin{abstract} Let $(R, \fm)$ be a Noetherian local ring and $M$ a finitely generated $R$-module of dimension $d$. A famous result of Northcott says that if $M$ is Cohen-Macaulay, then the index of reducibility of parameter ideals on $M$ is an invariant of the module. The aim of this paper is to extend Northcott's theorem for any finitely generated $R$-module. We call this invariant the stable value of the indices of reducibility of parameter ideals of $M$. We also introduce the limit value of the indices of reducibility of parameter ideals of $M$.
\end{abstract}

\section{Introduction}
Throughout this paper, let $(R, \fm, k)$ be a Noetherian local ring and $M$ a finitely generated $R$-module of dimension $d$. A submodule $N$ of $M$ is called an {\it irreducible submodule} if $N$ can not be written as an intersection of two properly larger submodules of $M$. The number of irreducible components of an irredundant irreducible decomposition of $N$, which is independent of the choice of the decomposition  by E. Noether \cite{N21}, is called the {\it index of reducibility} of $N$ and denoted by $\mathrm{ir}_M(N)$. Let $\fq$ be a parameter ideal of $M$ then {\it the index of reducibility} of $\fq$
on $M$ is the index of reducibility of $\fq M$ and denoted by $\mathrm{ir}_M(\fq)$. We have $\mathrm{ir}_M(\fq) = \dim_k \mathrm{Soc}(M/\fq M)$. Let $\widehat{R}$ and $\widehat{M}$ are the $\fm$-adic completions of $R$ and $M$, respectively. It is clear that $\mathrm{ir}_M(\fq) = \mathrm{ir}_{\widehat{M}}(\fq \widehat{R})$. Therefore, without loss of generality we always assume that $R$ is a homomorphic image of a Cohen-Macaulay local ring.

If $M$ is Cohen-Macaulay, then Northcott proved that $\mathrm{ir}_M(\fq)$ is an invariant of the module (cf. \cite{No57}) and it is called the {\it Cohen-Macaulay type} of $M$. More precisely, we have $\mathrm{ir}_M(\fq) = \dim_k \mathrm{Soc}(H^d_{\fm}(M))$ for all parameter ideals $\fq$, where $H^i_{\fm}(M)$ is the $i$-th local cohomology module of $M$. After that several researchers tried to extend Northcott's result for other classes of modules. For example Goto et al. studied the problem for Buchsbaum modules in \cite{GS03, GS84}; Truong and the authors extended the results for generalized Cohen-Macaulay modules in \cite{CQ11, CT08}. In more details, if $M$ is generalized Cohen-Macaulay it is proved in \cite[Corollary 4.3]{CQ11} and \cite[Theorem 1.1]{CT08} that $\mathrm{ir}_M(\fq)$ does not depend of the choice of parameter ideals contained in a large enough power of $\frak m$, and we have
$$\mathrm{ir}_M(\frak q) = \sum_{i=0}^d \binom{d}{i} \dim_k \mathrm{Soc}(H^i_{\fm}(M)). $$
Recently, the above results were extended for the classes of sequentially (generalized) Cohen-Macaulay modules and for parameter ideals generated by {\it good} systems of parameters in \cite{Q12,T13}. The main aim of this paper is to generalize Northcott's result for an arbitrary finitely generated $R$-module $M$. We prove that $\mathrm{ir}_M(\fq)$ is independent of the choice of parameter ideals which generated by certain systems of parameters, namely $C$-system of parameters (see Section 2 for details).
Our result is as follows.

\begin{theorem}[Theorem \ref{T3.2}]
Let $\underline{x} = x_1, ..., x_d$ be a $C$-system of parameters of $M$. Then the index of reducibility of $(\underline{x})$ on $M$, $\mathrm{ir}_M(\underline{x})$, is independent of the choice of $\underline{x}$.
\end{theorem}

We call the above invariant the {\it stable value} of the indices of reducibility of parameter ideals. This invariant agrees with previous invariants of the classes of (sequentially) (generalized) Cohen-Macaulay modules, and can be used to characterize these classes of rings (cf. Theorem \ref{T3.4}). We also introduce another invariant, say the {\it limit value} of the indices of reducibility of parameter ideals. The limit value can be though of as the infimum limit of the indices of reducibility of parameter ideals. It should be noted that we can not expect to have the supremum limit of the indices of reducibility of parameter ideals in general by \cite[Example 3.9]{GS84}.

This paper is organized as follows. In the next section we collect known results about the index of reducibility of parameter ideals and recall the notion of $C$-system of parameters. In Section 3, we study the stable value. The existence of the limit value is shown in Section 4. In the last section we compute the stable value for unmixed modules of dimension three and of depth two.

\begin{acknowledgement} The second author is grateful to Professor Shiro Goto for his comments which lead the main result of Section 5. The authors would like to thank the anonymous referee for helpful comments on the earlier version.
\end{acknowledgement}

\section{Preliminaries}
We start with the notion of annihilator of local cohomology which plays the key role of this paper.
\begin{notation} \rm Let $(R, \frak m)$ be a Noetherian local ring and $M$ a finitely generated $R$-module of dimension $d>0$.
\begin{enumerate}[{(i)}]
\item For all $i < d$ we set  $\frak a_i(M) =
\mathrm{Ann}H^{i}_\mathfrak{m}(M)$, and $\frak a(M) = \frak
a_0(M)...\frak a_{d-1}(M)$.
\item Put $\mathfrak{b}(M) = \bigcap_{\underline{x};i=1}^d
\mathrm{Ann}(0:x_i)_{M/(x_1,...,x_{i-1})M}$ where $\underline{x} =
x_1, ..., x_d$ runs over all systems of parameters of $M$.
\end{enumerate}
\end{notation}
\begin{remark} \label{R2.2}\rm
\begin{enumerate}[{(i)}]
\item Schenzel (cf.  \cite[Satz 2.4.5]{Sch82}) proved that
$$\mathfrak{a}(M) \subseteq \mathfrak{b}(M) \subseteq \mathfrak{a}_0(M) \cap \cdots \cap \mathfrak{a}_{d-1}(M).$$
\item If $R$ is a homomorphic image of a Cohen-Macaulay local ring, then $\dim R/\mathfrak{a}_i(M) \leq
i$ for all $i< d$. Furthermore, $\dim R/\mathfrak{a}_i(M) = i$ iff there exists $\frak p
\in \mathrm{Ass}M$ such that $\dim R/\frak p = i$ (see \cite[Theorem
8.1.1]{BH98}). Therefore we can choose a parameter element $x \in \frak a(M) \subseteq \frak b(M)$ of $M$.
\end{enumerate}
\end{remark}

We next mention the main object of this paper.
\begin{definition}\rm
Let $\frak q$ be a parameter ideal of $M$. The {\it index of
reducibility} of $\frak q$ on $M$ is the number of irreducible
components appear in an irredundant irreducible decomposition of
$\frak qM$, and denoted by $\mathrm{ir}_M(\frak q)$.
\end{definition}
\begin{remark}\label{R2.4} \rm
\begin{enumerate}[{(i)}]
\item It is well known that $\mathrm{ir}_M(\frak q) = \dim_{R/\frak
m}\mathrm{Soc}(M/\frak qM)$ (cf. \cite[Section 3]{No57}), where $\mathrm{Soc}(N) = 0:_N\frak m
\cong \mathrm{Hom}(R/\frak m, N)$ for an arbitrary $R$-module $N$.
\item A classical result of Northcott \cite{No57} says that the index of reducibility of parameter ideal of a Cohen-Macaulay module is an invariant of the module. In fact, if $M$ is Cohen-Macaulay, then $\mathrm{ir}_M(\frak q) = \dim_{R/\frak
m}\mathrm{Soc}(H^d_{\frak m}(M))$ for all parameter ideals $\frak q$.
\item Goto and Sakurai extended the Northcott result for Buchsbaum modules in \cite{GS03}. Truong and the fisrt author generalized Goto-Sakurai's result for generalized Cohen-Macaulay modules in \cite{CT08} (see also \cite[Corollary 4.3]{CQ11}). Let $M$ be a generalized Cohen-Macaulay module and $n_0$ a positive integer such that $\frak m^{n_0}H^i_{\frak
m}(M) = 0$ for all $i = 0,...,d-1$. Then for all parameter ideals $\frak
q$ contained in $\frak m^{2n_0}$ we have
$$\mathrm{ir}_M(\frak q) = \sum_{i=0}^{d} \binom{d}{i} \dim_{R/\frak m}\mathrm{Soc}(H^i_{\frak m}(M)).$$
\item In general, the index of reducibility of parameter ideals are not bounded above (\cite[Example 3.9]{GS84}). However, if $M$ is generalized Cohen-Macaulay, then Goto and Suzuki proved that
$$\mathrm{ir}_M(\frak q) \leq \sum_{i=0}^{d-1} \binom{d}{i} \ell(H^i_{\frak m}(M)) +
\dim_{R/\frak m}\mathrm{Soc}(H^d_{\frak m}(M))$$ for all parameter
ideals $\frak q$ of $M$ (cf. \cite[Theorem 2.1]{GS84}). Recently, the second author extended this result for the class of module with $\dim R/\frak a(M) \le 1$ in \cite{Q13}.
\end{enumerate}
\end{remark}
Next we recall briefly some basic facts about filtrations satisfying
the dimension condition and good system of parameters (cf. \cite{CC07,CN03}).
\begin{definition}\rm
\end{definition}
\begin{enumerate}[{(i)}]
\item We say that a finite filtration
$$\mathcal{F}: M_0 \subseteq M_1 \subseteq \cdots
\subseteq M_t = M$$ of submodules of $M$ satisfies the {\it
dimension condition} if  $ \dim M_0 < \dim M_1 < \cdots < \dim M_t$,
and then $\mathcal{F}$ is said to have length $t$. For
convenience, we always consider that $\dim M_1 > 0$.
\item A filtration of submodules $\mathcal{D} : D_0 \subseteq D_1 \subseteq
\cdots \subseteq D_t = M$ of $M$ is called the {\it dimension
filtration}
of $M$ if the following two conditions are satisfied:\\
(a) $D_{i-1}$ is the largest submodule of $D_i$ with $\dim
D_{i-1} < \dim D_i$ for $i = t, t-1, ...,1$.\\
(b) $D_0 = H^0_{\frak m}(M)$ is the 0-th local cohomology module of
$M$ with respect to the maximal ideal $\frak m$.
\end{enumerate}
\begin{definition}\rm Let $\mathcal{F}: M_0 \subseteq M_1 \subseteq \cdots
\subseteq M_t = M$ be a filtration satisfying the dimension
condition. Put $d_i = \dim M_i$. A system of parameters. $\underline{x} =
x_1,...,x_d$ of $M$ is called a {\it good system of parameters with
respect to} $\mathcal{F}$ if $M_i \cap (x_{d_i+1},...,x_d)M = 0$ for
$i= 0, 1, ..., t-1$. A good system of parameters with respect to the dimension
filtration is simply called a good system of parameters of $M$.
\end{definition}
\begin{definition}\rm
Let $\mathcal{F}: M_0 \subseteq M_1 \subseteq \cdots \subseteq M_t =
M$ be a filtration of submodules of $M$. Then, $\mathcal{F}$ is
called a {\it (generalized) Cohen-Macaulay filtration} if it satisfies
the dimension condition, $\dim M_0 = 0$ and $M_1/M_0,...,
M_t/M_{t-1}$ are (generalized) Cohen-Macaulay modules. Moreover, $M$
is called a {\it sequentially (generalized) Cohen-Macaulay module} if
it has a (generalized) Cohen-Macaulay filtration.
\end{definition}
It is clear that a (generalized) Cohen-Macaulay module is a sequentially (generalized) Cohen-Macaulay module with the (generalized) Cohen-Macaulay filtration $\mathcal{F}: 0 \subseteq M$. Truong in \cite{T13} and the second author in \cite{Q12} studied the stable value of the indices of reducibility of good parameter ideals of sequentially (generalized) Cohen-Macaulay modules.
\begin{theorem}[\cite{Q12}, Theorem 3.12] \label{T2.8} Let $M$ be a sequentially generalized Cohen-Macaulay module with a generalized Cohen-Macaulay filtration $\mathcal{F}: M_0 \subseteq M_1 \subseteq \cdots \subseteq M_t =
M$, $d_i = \dim M_i$ for all $i = 0, \ldots, t$. Then for every good system of parameters $\underline{x} = x_1,...,x_d$ of $M$ with
respect to $\mathcal{F}$ contained in a large enough power of $\frak m$, the index of reducibility of
$(\underline{x})$ on $M$ is independent of the choice of
$\underline{x}$ and
$$ \mathrm{ir}_M(\underline{x}) = \dim_{R/\frak m}\mathrm{Soc}(H^{0}_\mathfrak{m}(M)) +
\sum_{i=0}^{t-1}\sum_{j=1}^{d_{i+1}} \bigg(\binom{d_{i+1}}{j} -
\binom{d_{i}}{j} \bigg) \dim_{R/\frak
m}\mathrm{Soc}(H^{j}_\mathfrak{m}(M/M_i)).$$
\end{theorem}
\begin{corollary}[\cite{T13}] Let $M$ be a sequentially Cohen-Macaulay module of
dimension $d$. Then there is a positive integer $n$ such that for
every good system of parameters $\underline{x} = x_1,...,x_d$ of $M$ contained in
$\frak m^n$ the index of reducibility $\mathrm{ir}_M(\underline{x})$ is
independent of the choice of $\underline{x}$ and
$$\mathrm{ir}_M(\underline{x}) = \sum_{i=0}^d \dim_{R/\frak m}\mathrm{Soc}(H^{i}_\mathfrak{m}(M)).$$
\end{corollary}
Finally we present the splitting property of local cohomology in \cite{CQ11,CQ16}, and the notion of $C$-system of parameters.
\begin{definition}\rm The largest submodule of $M$ of dimension less than $d$ is called {\it the unmixed component} of $M$ and denoted by $U_M(0)$.
\end{definition}
Notice that if $\mathcal{D}: D_0 \subseteq D_1 \subseteq \cdots \subseteq D_t =
M$ is the dimension filtration of $M$, then $D_{t-1} = U_M(0)$. Moreover, if $x \in \frak b(M)$ is a parameter element of $M$, then $0:_M x = U_M(0)$.
\begin{theorem}[\cite{CQ16}, Corollary 3.2.5]\label{T2.11} Let $x \in \mathfrak{b}(M)^3$ be a parameter element of $M$. Let $U_M(0)$ be the unmixed component of $M$ and set $\overline{M} = M/U_M(0)$.
 Then
$$H^i_{\mathfrak{m}}(M/xM) \cong H^i_{\mathfrak{m}}(M) \oplus H^{i+1}_{\mathfrak{m}}(\overline{M})$$
for all $i<d-1$, and
$$\mathrm{Soc}(H^{d-1}_{\mathfrak{m}}(M/xM))  \cong \mathrm{Soc}(H^{d-1}_{\mathfrak{m}}(M)) \oplus
\mathrm{Soc}(H^{d}_{\mathfrak{m}}(M)).$$
\end{theorem}
It is natural to raise the following notion.
\begin{definition}[\cite{CQ16}] \rm A system of parameters $x_1, ..., x_d$ is called a {\it $C$-system of parameters} of $M$ if $x_d \in \mathfrak b(M)^3$ and $x_i \in \mathfrak b(M/(x_{i+1}, ..., x_d)M)^3$ for all $i = d-1, ..., 1$.
\end{definition}
 According to \cite{CC15}, $R$ is a homomorphic image of a Cohen-Macaulay local ring if and only if every finitely generated $R$-module admits a $C$-system of parameter.
The following are useful properties of $C$-system of parameters.
\begin{lemma}\label{L2.15} Let $x_1, ..., x_d$ be a $C$-system of parameters of $M$. Then
   \begin{enumerate}[{(i)}]\rm
\item $x_1, ..., x_d$ is a $d$-sequence.
\item $x_1^{n_1}, ..., x_d^{n_d}$ is a $C$-system of parameters of $M$ for all $n_1, ...., n_d \ge 1$.
\item For all $i \le d$ we have $x_1, ..., x_{i-1}, x_{i+1}, ..., x_d$ is a $C$-system of parameters of $M/x_iM$.
\item Let $N \subseteq H^0_{\mathfrak m}(M)$ be a submodule of finite length. Then $x_1, ..., x_d$ is a $C$-system of parameters of $M/N$.
\end{enumerate}
\end{lemma}
\begin{proof} For the three first properties see \cite[Lemma 3.8, Corollary 4.7 and Proposition 4.8]{CQ16}. For the last one we refer to \cite[Lemma 2.16 (iv)]{MQ16}.
\end{proof}
Let $x \in \mathfrak{b}(M)^3$ be a parameter element of $M$. The short exact sequence
$$0 \to M/U_M(0) \overset{x}{\to} M \to M/xM \to 0$$
induces
$$\cdots \to H^{d-1}_{\frak m}(M/xM) \to H^d_{\frak m}(M/U_M(0)) \cong H^d_{\frak m}(M) \to \cdots.$$
The last isomorphism of Theorem \ref{T2.11} says that the induced map on the socles is surjective
$$\mathrm{Soc}(H^{d-1}_{\frak m}(M/xM) \twoheadrightarrow  \mathrm{Soc}(H^d_{\frak m}(M)).$$
By induction we have the following, here $ H^i(\underline{x};M)$ denotes the $i$-th Koszul cohomology of $M$ with respect to the sequence $\underline{x} = x_1, ..., x_d$.
\begin{lemma}\label{L2.13} Let $\underline{x} = x_1, ..., x_d$ be a $C$-system of parameters of $M$. Then the canonical map
$$M/(\underline{x})M \cong H^d(\underline{x};M) \to H^d_{\frak m}(M)$$
induces a surjective map on the socles
$$\mathrm{Soc}(M/(\underline{x})M) \twoheadrightarrow \mathrm{Soc}(H^d_{\frak m}(M)).$$
\end{lemma}
We will need the following in the sequel.
\begin{lemma} \label{L2.14}
Let $\underline{x} = x_1, ..., x_d$ be a $C$-system of parameters of $M$. Then we have a surjective map on the socles
$$\mathrm{Soc}( H^{d-1}_{\frak m}(M/x_1M)) \twoheadrightarrow \mathrm{Soc}(H^d_{\frak m}(M)).$$
\end{lemma}
\begin{proof} We proceed by induction on $d$. The case $d=1$ is clear since $x_1 \in \frak b(M)^3$. For $d>1$, set $N  = H^0_{\frak m}(M)$. By Lemma \ref{L2.15} (iv), $\underline{x} = x_1, ..., x_d$ is a $C$-system of parameters of $M/N$. Moreover, we have $H^d_{\frak m}(M) \cong H^d_{\frak m}(M/N)$ and $H^{d-1}_{\frak m}(M/x_1M) \cong H^{d-1}_{\frak m}(M/x_1M+N)$. Thus we can assume henceforth that $\mathrm{depth}(M)>0$. Hence $x_1$ is $M$-regular since $\underline{x}$ is a $d$-sequence (see Lemma \ref{L2.15} (i)). The short exact sequence
$$0 \to M \overset{x_1}{\to} M \to M/x_1M \to 0$$
induces the following commutative diagram
$$
\begin{CD}
\cdots @>>> H^{d-1}(\underline{x};M) @>>> H^{d-1}(\underline{x};M/x_1M) @>>> H^{d}(\underline{x};M) @>>> \cdots\\
@. @VVV @VVV @VVV \\
\cdots @>>> H^{d-1}_{\frak m}(M) @>>> H^{d-1}_{\frak m}(M/x_1M) @>>> H^{d}_{\frak m}(M) @>>> \cdots\\
\end{CD}
$$
with rows are exact. Furthermore the top row splits by \cite[Lemma 1.7]{GS84}. Restricting on the socle modules we have the commutative diagram
$$
\begin{CD}
0 @>>> \mathrm{Soc}(H^{d-1}(\underline{x};M)) @>>> \mathrm{Soc}(H^{d-1}(\underline{x};M/x_1M)) @>\delta>> \mathrm{Soc}(H^{d}(\underline{x};M)) @>>> 0\\
@. @VVV @VVV @V\pi VV \\
\cdots @>>> \mathrm{Soc}(H^{d-1}_{\frak m}(M)) @>>> \mathrm{Soc}(H^{d-1}_{\frak m}(M/x_1M)) @>\gamma>> \mathrm{Soc}(H^{d}_{\frak m}(M)) @>>> \cdots.\\
\end{CD}
$$
Notice that $\delta$ is surjective and $\pi$ is also surjective by Lemma \ref{L2.13}. Thus the map
$$\gamma : \mathrm{Soc}(H^{d-1}_{\frak m}(M/x_1M)) \to \mathrm{Soc}(H^{d}_{\frak m}(M))$$
is surjective.
\end{proof}

\section{The stable value}
In this section we cover all stable results on the index of reducibility of parameter ideals mentioned in the previous section. The following lemma immediately follows from Theorem \ref{T2.11}.
\begin{lemma}\label{L3.1}
Let $x \in \mathfrak{b}(M)^3$ be a parameter element of $M$. Let $U_M(0)$ be the unmixed component of $M$ and set $\overline{M} = M/U_M(0)$.
 Then for all $i\leq d-1$ we have
$$\dim_{R/\frak m} \mathrm{Soc}(H^i_{\mathfrak{m}}(M/xM)) = \dim_{R/\frak m} \mathrm{Soc}(H^i_{\mathfrak{m}}(M)) +
\dim_{R/\frak m} \mathrm{Soc}(H^{i+1}_{\mathfrak{m}}(\overline{M})).$$
\end{lemma}

We now state the main result of this section.
\begin{theorem}\label{T3.2}
Let $\underline{x} = x_1, ..., x_d$ be a $C$-system of parameters of $M$. Then the index of reducibility of $\underline{x}$ on $M$, $\mathrm{ir}_M(\underline{x})$, is independent of the choice of $\underline{x}$.
\end{theorem}
\begin{proof}
Let $\underline{y} = y_1, ..., y_d$ be a $C$-system of parameters of $M$ we prove that $\mathrm{ir}_M(\underline{x}) = \mathrm{ir}_M(\underline{y})$. We proceed by induction of $d$. The case $d=1$ follows from Lemma \ref{L3.1}. For $d>1$ we choose $z \in \frak b(M/(x_2, \ldots, x_d)M)^3 \cap \frak b(M/(y_2, \ldots, y_d)M)^3$ is a parameter element of both $M/(x_2, \ldots, x_d)M$ and $M/(y_2, \ldots, y_d)M$. The existence of $z$ follows from Remark \ref{R2.2}. Therefore $\underline{x}' = z, x_2, \ldots,x_d$ and $\underline{y}' = z, y_2, \ldots, y_d$ are $C$-systems of parameters of $M$. By Lemma \ref{L2.15} both $x_1, \ldots, x_{d-1}$ and $z,x_2, \ldots, x_{d-1}$ are $C$-systems of parameters of $M/x_dM$. We have
$$\mathrm{ir}_M(\underline{x}) = \mathrm{ir}_{M/x_dM}(x_1, \ldots, x_{d-1}) = \mathrm{ir}_{M/x_dM}(z,x_2, \ldots, x_{d-1}) = \mathrm{ir}_M(\underline{x}'),$$
where the second equality follows from the induction.
Similarly, we have
$$\mathrm{ir}_M(\underline{y}) = \mathrm{ir}_M(\underline{y}').$$
On the other hand, $\underline{x}'' = x_2, \ldots,x_d$ and $\underline{y}'' = y_2, \ldots, y_d$ are $C$-systems of parameters of $M/zM$ by Lemma \ref{L2.15} again. Applying the inductive hypothesis for $M/zM$ we have
$$\mathrm{ir}_M(\underline{x}') = \mathrm{ir}_{M/zM}(\underline{x}'') = \mathrm{ir}_{M/zM}(\underline{y}'') = \mathrm{ir}_M(\underline{y}').$$
The proof is complete.
\end{proof}
\begin{definition}\rm We denote the invariant in the previous theorem by $\mathcal{N}_R(M)$. We call $\mathcal{N}_R(M)$ the {\it stable value} of the indices of reducibility of parameter ideals of $M$.
\end{definition}

The stable value $\mathcal{N}_R(M)$ can be used to characterize (sequentially) Cohen-Macaulay modules as follows.
\begin{theorem}\label{T3.4} The following hold true.
\begin{enumerate}[{(i)}]\rm
\item {\it $M$ is Cohen-Macaulay if and only if $\mathcal{N}_R(M) = \dim_{R/\frak m}\mathrm{Soc}(H^d_{\frak m}(M))$}.
\item {\it $\mathcal{N}_R(M)\ge \sum_{i=0}^{d}\dim_{R/\frak m}\mathrm{Soc}(H^i_{\frak m}(M))$ and the equality occurs if and only if $M$ is sequentially Cohen-Macaulay}.
\end{enumerate}
\end{theorem}
\begin{proof} (i) follows from (ii).\\
(ii) Induction on $d$. The case $d=1$ is not difficult since $M$ is generalized Cohen-Macaulay. For $d>1$, let $\mathcal{D}: D_0 \subseteq D_1 \subseteq \cdots \subseteq D_t =
M$ be the dimension filtration of $M$. Notice that $D_{t-1} = U_M(0)$ the unmixed component of $M$. Let $\underline{x} = x_1, \ldots, x_d$ be a $C$-system of parameters of $M$. By Lemma \ref{L3.1} we have
$$\sum_{i=0}^{d-1}\dim_{R/\frak m}\mathrm{Soc}(H^i_{\frak m}(M/x_dM)) = \sum_{i=0}^{d}\dim_{R/\frak m}\mathrm{Soc}(H^i_{\frak m}(M)) + \sum_{i=0}^{d-1}\dim_{R/\frak m}\mathrm{Soc}(H^i_{\frak m}(M/D_{t-1})).$$
Thus $\sum_{i=0}^{d-1}\dim_{R/\frak m}\mathrm{Soc}(H^i_{\frak m}(M/x_dM)) \ge \sum_{i=0}^{d}\dim_{R/\frak m}\mathrm{Soc}(H^i_{\frak m}(M))$ and the equality occurs if and only if $M/D_{t-1}$ is Cohen-Macaulay. Since $x_1, \ldots, x_{d-1}$ is also a $C$-system of parameters of $M/x_dM$, we have
$$\mathcal{N}_R(M) = \mathcal{N}_R(M/x_dM) \ge \sum_{i=0}^{d-1}\dim_{R/\frak m}\mathrm{Soc}(H^i_{\frak m}(M/x_dM)) \ge \sum_{i=0}^{d}\dim_{R/\frak m}\mathrm{Soc}(H^i_{\frak m}(M)),$$
where the first inequality follows from the induction. Therefore $\mathcal{N}_R(M) \ge \sum_{i=0}^{d}\dim_{R/\frak m}\mathrm{Soc}(H^i_{\frak m}(M))$ and the equality occurs if and only if $M/x_dM$ is sequentially Cohen-Macaulay and $M/D_{t-1}$ is Cohen-Macaulay. The last condition is equivalent to $M$ is sequentially Cohen-Macaulay by the following claim.\\
{\bf Claim.} $M$ is sequentially Cohen-Macaulay if and only if $M/x_dM$ is sequentially Cohen-Macaulay and $M/D_{t-1}$ is Cohen-Macaulay.\\
{\it Proof of Claim.} "$\Rightarrow$" follows from \cite[Lemma 6.4]{CC07-1} and the definition of sequentially Cohen-Macaulay module. \\
"$\Leftarrow$", since $\overline{M} :=M/D_{t-1}$ is Cohen-Macaulay we have the following short exact sequence
$$0 \to D_{t-1} \to M/x_dM \to \overline{M}/x_d\overline{M} \to 0.$$
Thus the dimension filtration of $M/x_dM$ is either
$$D_0 \cong \frac{D_0+x_dM}{x_dM} \subseteq \cdots \subseteq D_{t-1} \cong \frac{D_{t-1}+x_dM}{x_dM} \subseteq M/x_dM,$$
if $\dim D_{t-1}<d-1$, or
$$D_0 \cong \frac{D_0+x_dM}{x_dM} \subseteq \cdots \subseteq D_{t-2} \cong \frac{D_{t-2}+x_dM}{x_dM} \subseteq M/x_dM,$$
if $\dim D_{t-1}=d-1$. If $\dim D_{t-1}<d-1$ we have $D_{i}/D_{i-1}$ is Cohen-Macaulay for all $i = 0, \ldots,t-1$ since $M/x_dM$ is sequentially Cohen-Macaulay. So $M$ is sequentially Cohen-Macaulay. Suppose $\dim D_{t-1} = d-1$, we need only to prove that $D_{t-1}/D_{t-2}$ is Cohen-Macaulay. This requirement can be proved by applying the local cohomology functor $H^i_{\frak m}(-)$ to the short exact sequence
$$0 \to D_{t-1}/D_{t-2} \to M/(D_{t-2}+x_dM) \to \overline{M}/x_d\overline{M} \to 0.$$
Therefore the claim is proved and the proof is complete.
\end{proof}
\begin{remark}\label{R3.5}\rm If $M$ is (sequentially) generalized Cohen-Macaulay modules, then it is easily seen that $\mathcal{N}_R(M)$ is just the invariants of Remark \ref{R2.4} (iii) and Theorem \ref{T2.8}. At the time of writing the authors were not able to use these equalities to characterize these classes of modules.

\end{remark}

\section{The limit value}
For each $n \ge 1$ we define
$$\alpha_n(M) = \min \{\mathrm{ir}_M(\frak q) \,|\, \frak q \subseteq \frak m^n \text{ a system of parameters of } M\}.$$
We have $\{\alpha_n(M)\}_{n\ge 1}$ is a non-decreasing sequence and bounded above by $\mathcal{N}_R(M)$ since we can choose a $C$-system of parameters which is contained in $\frak m^n$ for all $n \ge 1$. Thus the following is well defined.
\begin{definition} The limit $\alpha(M) = \lim \alpha_n(M)$ exists. We call $\alpha(M)$ the {\it limit value} of the indices of reducibility of parameter ideals of $M$
\end{definition}
We can use the limit value to characterize the Cohen-Macaulayness as follows.
\begin{theorem} For any local ring $(R, \frak m)$ we have 
$$\dim_{R/\frak m}\mathrm{Soc}(H^d_{\frak m}(M)) \le \alpha(M) \le \mathcal{N}_R(M).$$ 
Moreover $\alpha(M) = \dim_{R/\frak m}\mathrm{Soc}(H^d_{\frak m}(M))$ if and only if $M$ is Cohen-Macaulay.
\end{theorem}
\begin{proof} The last inequality follows from the above discussion. The first inequality follows from of a result of Goto and Sakurai \cite[Lemma 3.12]{GS03} which claims that for all parameter ideals contained in a large enough power of $\frak m$ the canonical map
$$M/\frak qM = H^d(\frak q; M) \to H^d_{\frak m}(M)$$
induced a surjective map on the socle modules $\mathrm{Soc}(M/\frak qM) \twoheadrightarrow \mathrm{Soc}(H^d_{\frak m}(M))$.
The last assertion follows from \cite[Theorem 5.2]{CQT15}.
\end{proof}
By Remark \ref{R2.4} (iii) we have the following.
\begin{proposition} Let $M$ be a generalized Cohen-Macaulay module. Then we have
  $$\alpha(M) = \sum_{i=0}^{d} \binom{d}{i} \dim_{R/\frak m}\mathrm{Soc}(H^i_{\frak m}(M)).$$
\end{proposition}
The number $\alpha(M)$ is still mysterious to us for non-generalized Cohen-Macaulay modules. Below we compute this invariant in a very simple example.
\begin{example}
Let $R = k[[X,Y,Z]]/(XY, XZ) = K[[x,y,z]]$, where $K$ is a field. We have $R$ is a sequentially Cohen-Macaulay ring of dimension two with the Cohen-Macaulay filtration $0 \subseteq (x) \subseteq R$. Let $\frak q$ be a parameter ideal of $R$. Since $\overline{R} = R/(x)$ is Cohen-Macaulay, the short exact sequence
$$0 \to (x) \to R \to \overline{R} \to 0$$ 
induces the short exact sequence
$$0 \to (x)/x\frak q \to R/\frak q \to \overline{R}/\frak q \overline{R} \to 0.$$
Therefore we have the left exact sequence of the socle modules
$$0 \to \mathrm{Soc}((x)/x\frak q) \to \mathrm{Soc}(R/\frak q) \to \mathrm{Soc}(\overline{R}/\frak q \overline{R}).$$
Notice that $(x) \cong K[[X]]$ is a DVR and $\overline{R} \cong K[[Y,Z]]$ is a regular local ring. Hence
$$\mathrm{ir}(\frak q) \le \dim_K\mathrm{Soc}((x)/x\frak q) + \dim_K\mathrm{Soc}(\overline{R}/\frak q \overline{R}) = 2$$
for all $\frak q$. We use again \cite[Lemma 3.12]{GS03} to obtain that the natural map 
$$\mathrm{Soc}(R/\frak q) \to \mathrm{Soc}(H^2_{\frak m}(R))$$
is surjective for any $\frak q$ contained in a large enough power of $\frak m$.
On the other hand, since $\overline{R}$ is Cohen-Macaulay we have
$$\mathrm{Soc}(\overline{R}/\frak q \overline{R}) \cong \mathrm{Soc}(H^2_{\frak m}(\overline{R})) \cong \mathrm{Soc}(H^2_{\frak m}(R)).$$ 
Therefore the natural map $\mathrm{Soc}(R/\frak q) \to \mathrm{Soc}(\overline{R}/\frak q \overline{R})$ is surjective for all $\frak q$ contained in a large enough power of $\frak m$. Hence $\alpha(R) =2$.
\end{example}

We list here some natural questions about the limit value.
\begin{question} Does $\alpha(M) = \mathcal{N}_R(M)$ for every finitely generated $R$-module $M$?
\end{question}
\begin{question} For every finitely generated $R$-module $M$,
does $\sum_{i=0}^d\dim_{R/\frak m}\mathrm{Soc}(H^i_{\frak m}(M)) \le \alpha(M)$, and the equality occurs if and only if $M$ is sequentially Cohen-Macaulay?
\end{question}

\section{The three dimensional case}
In this section we compute the stable value $\mathcal{N}_R(M)$ when $M$ is unmixed, $\dim M =  3$ and $\mathrm{depth}(M)=2$. Note that if $M$ is generalized Cohen-Macaulay then it satisfies the Serre condition $S_2$. We also refer to \cite{HH94} for the $S_2$-ification.
\begin{lemma}\label{L5.1} Let $S$ be the $S_2$-ification of $M$. Suppose $M$ is not generalized Cohen-Macaulay. Then $D:=S/M$ is a Cohen-Macaulay module of dimension one.
\end{lemma}
\begin{proof} For each $\frak p \in \mathrm{Spec}(R)$ such that $M_{\frak p}$ is $S_2$ we have $S_{\frak p} = M_{\frak p}$ so $D_{\frak p} = 0$. Since $M$ is unmixed and $\dim M = 3$ we have $M_{\frak p}$ is Cohen-Macaulay for all $\frak p \in \mathrm{Supp}(M)$ such that $\dim R/\frak p \ge 2$. Therefore $\dim R/\frak p \le 1$ for all $\frak p \in \mathrm{Ass}D$. Moreover the short exact sequence
$$0 \to M \to S \to D \to 0$$
induces the exact sequence
$$0 \to H^0_{\frak m}(D) \to H^1_{\frak m}(M) \to \cdots.$$
But $\mathrm{depth}(M)  = 2$ so $H^0_{\frak m}(D) = 0$. Hence $\dim D = 1$ and $D$ is Cohen-Macaulay.
\end{proof}
Let $\underline{x} = x_1, x_2, x_3$ be a $C$-system of parameters of $M$. Then we have $M/x_1M$ is unmixed by $\mathrm{depth}(M)=2$ and \cite[Remark 3.3]{MQ16}. Thus $M/x_1M$ is generalized Cohen-Macaulay and $x_2, x_3$ is a $C$-system of parameters of $M/x_1M$.

\begin{remark}\label{R5.2}\rm In fact we have $\mathrm{Ass}D = \mathcal{F}(M)$ where
  $$\mathcal{F}(M) = \{\mathfrak{p}\in \mathrm{Spec}(R)\,|\, \dim M_{\mathfrak{p}}>
1=\mathrm{depth}M_{\mathfrak{p}},\, \mathfrak{p} \neq \mathfrak{m}
\}$$ is a finite set (cf. \cite[Lemma 3.2]{GN01}). Moreover $\mathcal{F}(M) = \{\frak p \in \mathrm{Ass}(M/x_3M)\,|\, \dim R/\frak p \le 1\}$ by \cite[Proposition 4.14]{CQ16}. Thus $x_1$ is a parameter element of $D$.
\end{remark}

The following is the main result of this section.
\begin{theorem}\label{T5.3} Let $M$ be an unmixed module of dimension $3$ and $\mathrm{depth}(M) = 2$. Let $S$ be the $S_2$-ification of $M$, we have
$$\mathcal{N}_R(M) = 2\dim_{R/\frak m}\mathrm{Soc}(H^2_{\frak m}(M)) + \dim_{R/\frak m}\mathrm{Soc}(H^3_{\frak m}(M)) + \dim_{R/\frak m}\mathrm{Soc}(H^2_{\frak m}(S)).$$
\end{theorem}
\begin{proof} If $M$ is generalized Cohen-Macaulay, we have $S = M$ and the assertion follows from Remark \ref{R2.4} (iii). Suppose $M$ is not generalized Cohen-Macaulay. Choose a $C$-system of parameters $\underline{x} = x_1, x_2, x_3$ of $M$ such that it is also a standard system of parameters of $S$. We have
$$\mathcal{N}_R(M) = \mathrm{ir}_M(\underline{x})=\mathrm{ir}_{M/x_1M}((x_2,x_3)) = \sum_{i=0}^2 \binom{2}{i} \dim_{R/\frak m}\mathrm{Soc}(H^i_{\frak m}(M/x_1M))$$
by Remarks \ref{R2.4} and \ref{R3.5}.\\
On the other hand, applying the local cohomology functor to the short exact sequence
$$0 \to M \overset{x_1}{\to} M \to M/x_1M \to 0$$
we have $H^0_{\frak m}(M/x_1M) = 0$ and $H^1_{\frak m}(M/x_1M) \cong 0:_{H^2_{\frak m}(M)}x_1$ and the exact sequence
$$0 \to H^2_{\frak m}(M)/x_1 H^2_{\frak m}(M) \to H^2_{\frak m}(M/x_1M) \to H^3_{\frak m}(M) \to \cdots.$$
Thus $\dim_{R/\frak m}\mathrm{Soc}(H^1_{\frak m}(M/x_1M)) = \dim_{R/\frak m}\mathrm{Soc}(H^2_{\frak m}(M))$. We need only to prove that
$$\dim_{R/\frak m}\mathrm{Soc}(H^2_{\frak m}(M/x_1M)) = \dim_{R/\frak m}\mathrm{Soc}(H^3_{\frak m}(M)) + \dim_{R/\frak m}\mathrm{Soc}(H^2_{\frak m}(S)).$$
By Lemma \ref{L2.14} we have the following exact sequence
$$0 \to \mathrm{Soc}(H^2_{\frak m}(M)/x_1 H^2_{\frak m}(M)) \to \mathrm{Soc}(H^2_{\frak m}(M/x_1M)) \to \mathrm{Soc}(H^3_{\frak m}(M)) \to 0.$$
Therefore the assertion follows from the following claim.\\
{\bf Claim.} $H^2_{\frak m}(M)/x_1 H^2_{\frak m}(M) \cong H^2_{\frak m}(S)$.\\
{\it Proof of Claim.} Since $D= S/M$ is a Cohen-Macaulay module of dimension one (cf. Lemma \ref{L5.1}), the short exact sequence
$$0 \to M \to S \to D \to 0$$
induces the exact sequence
$$0 \to H^1_{\frak m}(D) \to H^2_{\frak m}(M) \to H^2_{\frak m}(S) \to 0.$$
Consider the following commutative diagram
$$
\begin{CD}
0 @>>> H^1_{\frak m}(D) @>>> H^2_{\frak m}(M) @>>> H^2_{\frak m}(S) @>>> 0\\
@. @Vx_1VV @Vx_1VV @Vx_1VV \\
0 @>>> H^1_{\frak m}(D) @>>> H^2_{\frak m}(M) @>>> H^2_{\frak m}(S) @>>> 0.\\
\end{CD}
$$
By Remark \ref{R5.2}, $\dim D/x_1D = 0$ thus the map $H^1_{\frak m}(D) \overset{x_1}{\to} H^1_{\frak m}(D)$ is surjective. Moreover since $\underline{x}$ is a standard system of parameter of $S$, $x_1H^2_{\frak m}(S) = 0$. By the Snake lemma we have $H^2_{\frak m}(M)/x_1 H^2_{\frak m}(M) \cong H^2_{\frak m}(S)$. This proves the claim.\\
\noindent The proof is complete.
\end{proof}

\begin{example} Let $T = K[[a,b,c,d,e]]$ be the formal power series ring over a field $K$. Let $R = T/((a,b) \cap (c,d))$. It is easy to check that $R$ is unmixed of $\dim R = 3$ and $\depth (R) = 2$. The $S_2$-ification of $R$ is
$$S = T/(a,b) \oplus T/(c,d).$$
We have $S$ is Cohen-Macaulay and $D = S/R \cong T/(a,b,c,d)$. By the short exact sequence
$$0 \to R \to S \to D \to 0$$
one can compute that $\dim_K\mathrm{Soc}(H^2_{\frak m}(R)) = 1$ and $\dim_K\mathrm{Soc}(H^3_{\frak m}(R)) = 2$. Therefore $\mathcal{N}(R) = 4$.
\end{example}


\begin{thebibliography}{99}

\bibitem{BH98} W. Bruns and J. Herzog, \emph{Cohen-Macaulay rings}, Cambridge University Press (Revised
edition), 1998.

\bibitem{CC07-1} N.T. Cuong and D.T. Cuong, \emph{dd-sequences and partial Euler-Poincare characteristics of Koszul complex}, J. Algebra Appl. {\bf 6} (2007),  207--231.

\bibitem{CC07} N.T. Cuong and D.T. Cuong, \emph{On sequentially
Cohen-Macaulay modules}, Kodai Math. J. {\bf 30} (2007), 409--428.


\bibitem{CC15} N.T. Cuong and D.T. Cuong, \emph{Local cohomology annihilators and Macaulayfication}, Acta Math. Vietnam. {\bf 42} (2018), 37--60.

\bibitem{CN03} N.T. Cuong and L.T. Nhan, \emph{Pseudo Cohen-Macaulay and
pseudo generalized Cohen-Macaulay modules}, J. Algebra {\bf 267}
(2003), 156--177.

\bibitem{CQ11} N.T. Cuong and P.H. Quy, \emph{A splitting theorem for local cohomology and its applications}, J. Algebra {\bf 331} (2011),
512--522.

\bibitem{CQ16} N. T.  Cuong and  P. H. Quy, \emph{the structure of finitely generated modules over quotients of Cohen-Macaulay local rings}, Preprint, Arxiv: 1612.07638.

\bibitem{CQT15} N.T. Cuong, P.H. Quy and H.L. Truong, \emph{On the index of reducibility in Noetherian modules}, J. Pure Appl. Algebra {\bf 219} (2015),
4510--4520.

\bibitem{CT08} N.T. Cuong and H.L. Truong, \emph{Asymptotic behavior of parameter ideals in generalized Cohen-Macaulay
module}, {\it J. Algebra} {\bf 320} (2008), 158--168.


\bibitem{GN01} S. Goto and Y. Nakamura, \emph{Multiplicity and tight
closures of parameters}, J. Algebra {\bf 244} (2001), 302--311.

\bibitem{GS03} S. Goto and H. Sakurai, \emph{The equality $I^2 = QI$ in Buchsbaum rings}, Rend. Sem. Univ. Padova.
{\bf 110} (2003), 25--56.

\bibitem{GS84} S. Goto and N. Suzuki, \emph{Index of reducibility of parameter ideals in a local ring}, J. Algebra {\bf 87} (1984),
53--88.

\bibitem{HH94} M. Hochster and C. Huneke,
\emph{Indecomposable canonical modules and connectedness}, Commutative algebra: syzygies, multiplicities, and birational algebra (South Hadley, MA, 1992), Contemp. Math. {\bf 159} (1994), 197--208.

\bibitem{MQ16} M. Morales and P.H. Quy, \emph{A study of the length function of generalized fractions of modules}, Proc. Edinburgh Math. Soc. {\bf 60} (2017), 721--737.

\bibitem{N21} E. Noether, \emph{Idealtheorie in Ringbereichen}, Math. Ann. {\bf 83} (1921), 24--66.

\bibitem{No57} D.G. Northcott, \emph{On irreducible ideals in local rings}, J. London Math. Soc. {\bf 32} (1957), 82--88.

\bibitem{Q12} P.H. Quy, \emph{Asymptotic behaviour of good systems of parameters of sequentially generalized Cohen-Macaulay modules}, Kodai Math.
J. {\bf 35} (2012), 576--588.

\bibitem{Q13} P.H. Quy, \emph{On the uniform bound of the index of reducibility of parameter ideals of a module whose polynomial type is at most one}, Arch. Math. {\bf 101} (2013), 469--478.

\bibitem{Sch82} P. Schenzel, \emph{Dualisierende komplexe in der lokalen algebra und Buchsbaum - ringe},
Lect. Notes in Math., Springer-Verlag Berlin - Heidelberg - New
York, 1982.

\bibitem{T13} H.L. Truong, \emph{Index of reducibility of distinguished parameter ideals and sequentially Cohen-Macaulay modules}, Proc. Amer. Math. Soc. {\bf 141} (2013), 1971--1978.
\end{thebibliography}
\end{document}